\pdfoutput=1
\RequirePackage{ifpdf}
\ifpdf 
\documentclass[pdftex]{sigma}
\else
\documentclass{sigma}
\fi

\numberwithin{equation}{section}

\newtheorem{Theorem}{Theorem}[section]
\newtheorem*{Theorem*}{Theorem}

\newtheorem{Lemma}[Theorem]{Lemma}
\newtheorem{Proposition}[Theorem]{Proposition}
 { \theoremstyle{definition}
\newtheorem{Definition}[Theorem]{Definition}

\newtheorem{Remark}[Theorem]{Remark} }

\newcommand{\mnote}[1]
{\protect{\stepcounter{mnotecount}}$^{\mbox{\footnotesize
$
\bullet$\themnotecount}}$ \marginpar{
\raggedright\tiny\em
$\!\!\!\!\!\!\,\bullet$\themnotecount: #1} }

\newcommand{\C}{\mathbb{C}}
\newcommand{\Z}{\mathbb{Z}}
\newcommand{\PP}{\mathbb{P}}
\newcommand{\RP}{\mathbb{RP}}
\newcommand{\tL}{\widetilde{L}}
\newcommand{\tP}{\widetilde{P}}
\newcommand{\R}{\mathbb{R}}

\newcommand{\Rho}{\mathrm{P}}

\def\p{\partial}
\def\be{\begin{equation}}

\def\u{\mathfrak{u}}
\def\z{\mathfrak{z}}
\def\I{\mathcal{I}}

\def\ee{\end{equation}}

\def\bea{\begin{eqnarray}}
\def\eea{\end{eqnarray}}

\newcommand{\spp}{\mathbb{S}}

 \renewcommand{\Z}{\mathcal{Z}}
 \newcommand{\U}{\mathcal{U}}

\newcommand{\cI}{{\mathcal I}}
\newcommand{\st}{\mid }
\newcommand{\uh}{\widehat{u}}
\newcommand{\zh}{\widehat{z}}

\begin{document}
\allowdisplaybreaks

\newcommand{\arXivNumber}{2201.04717}

\renewcommand{\thefootnote}{}

\renewcommand{\PaperNumber}{027}

\FirstPageHeading

\ShortArticleName{Twistor Theory of Dancing Paths}

\ArticleName{Twistor Theory of Dancing Paths\footnote{This paper is a~contribution to the Special Issue on Twistors from Geometry to Physics in honor of Roger Penrose. The~full collection is available at \href{https://www.emis.de/journals/SIGMA/Penrose.html}{https://www.emis.de/journals/SIGMA/Penrose.html}}}

\Author{Maciej DUNAJSKI}

\AuthorNameForHeading{M.~Dunajski}

\Address{Department of Applied Mathematics and Theoretical Physics, University of Cambridge,\\ Wilberforce Road, Cambridge CB3 0WA, UK}
\Email{\href{mailto:m.dunajski@damtp.cam.ac.uk}{m.dunajski@damtp.cam.ac.uk}}
\URLaddress{\url{https://www.damtp.cam.ac.uk/user/md327/}}

\ArticleDates{Received January 14, 2022, in final form March 28, 2022; Published online March 31, 2022}

\Abstract{Given a path geometry on a surface $\mathcal{U}$, we construct a causal structure on a~four-manifold which is the configuration space of non-incident pairs (point, path) on $\mathcal{U}$. This causal structure corresponds to a conformal structure if and only if $\mathcal{U}$ is a real projective plane, and the paths are lines. We~give the example of the causal structure given by a~symmetric sextic, which corresponds on an ${\rm SL}(2,{\mathbb R})$-invariant projective structure where the paths are ellipses of area $\pi$ centred at the origin. We~shall also discuss a~causal structure on a~seven-dimensional manifold corresponding to non-incident pairs (point, conic) on a~projective plane.}


\Classification{32L25; 53A20}

\begin{flushright}
\begin{minipage}{60mm}
\it Dedicated to Roger Penrose\\ on the occasion of his 90th birthday
\end{minipage}
\end{flushright}

\renewcommand{\thefootnote}{\arabic{footnote}}
\setcounter{footnote}{0}

\vspace{1mm}

\section{Introduction}
It is a great honor and pleasure to be able to contribute to the celebrations
of Roger's 90th birthday. Roger is now twice as old as his
nonlinear graviton construction~\cite{Penrose}. When I~arrived in Oxford as a graduate student, it was
already known that the integrability of the nonlinear
solitonic systems has its roots in the Ward transform for the anti-self-dual Yang--Mills equa\-tions~\cite{mw, ward0}. I spent three years of my PhD
showing, under Lionel Mason's supervision, that the
curved space twistor theory behind the nonlinear graviton construction
underlies the integrability of the complementary class of the, so called,
dispersionless integrable systems~\cite{MD_phd}. Roger made several comments
to me about my work, but
on average it took me three months to understand his geometric insight,
and another three to realise that he was right. At that time Roger had already
moved to the notoriously difficult googly problem, as well as the interplay between gravity and quantum mechanics.
The photograph on Figure~\ref{Fig1} was taken at the GRG meeting in Pune, India in 1997, where Roger gave a public lecture about the latter.
\begin{figure}[h]\centering
\includegraphics[width=147mm]{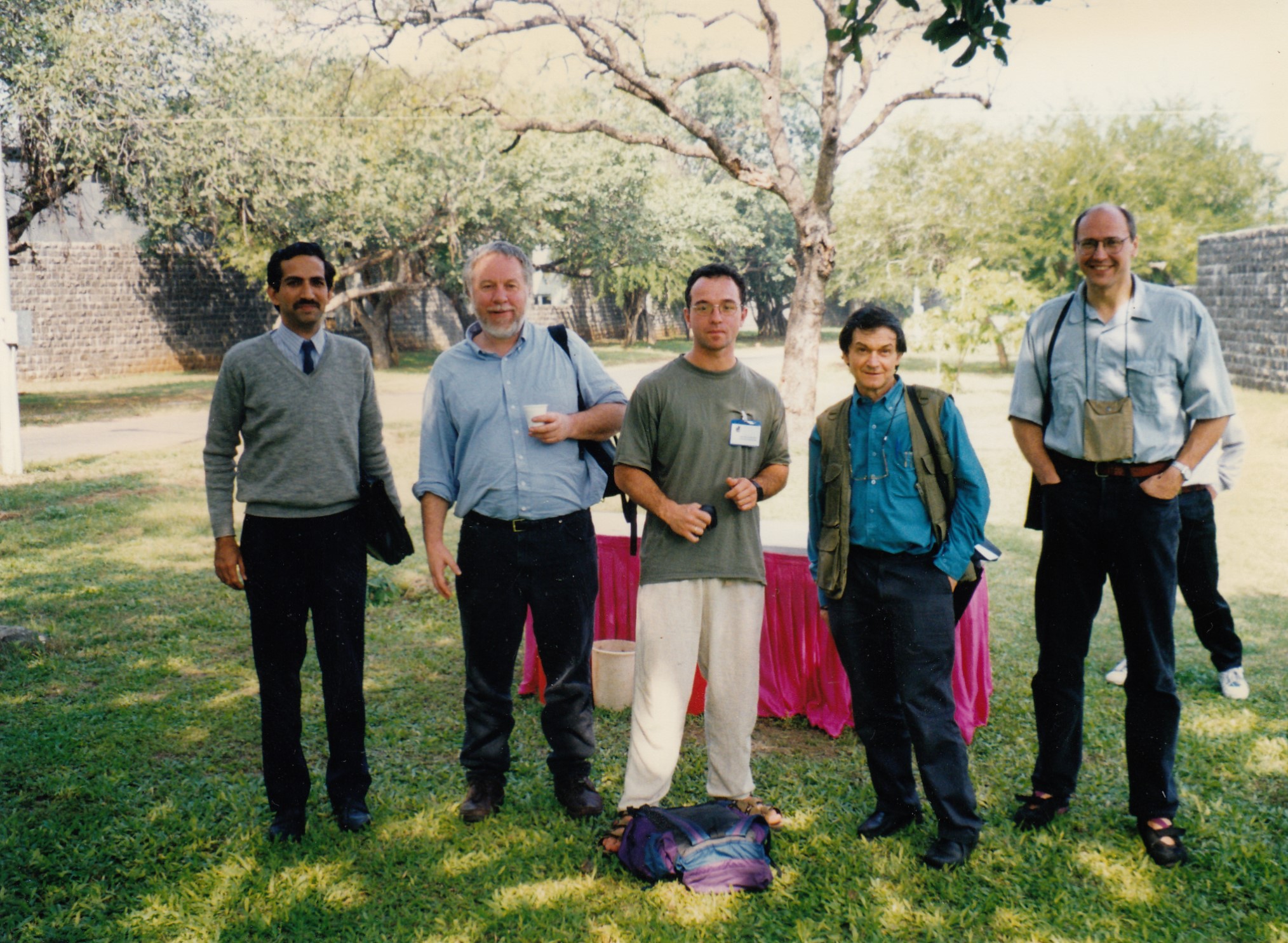}
\caption{Pune, December 1997. From the left: Devendra Kapadia, George Sparling, Maciej Dunajski, Ro\-ger Penrose, Jerzy Lewandowski.}\label{Fig1}
\end{figure}
The lecture proved to be so popular among the citizens of Pune that the lecture hall had to be sealed off by riot police, to prevent overcrowding.

The nonlinear graviton construction describes a link between
algebraic geometry of holomorphic curves in a complex three-fold (the twistor space), and differential geometry on the moduli space $M$ of these curves. Two curves in three
dimensions generically do not intersect (this is backed by a
geometric intuition based on curves in $\R^3$ which carries over
to the twistor space, as the curves are holomorphic). The intersection
condition defines a notion of null separation: two points on $M$ are null separated iff the corresponding curves in the twistor space intersect. All~curves intersecting a fixed curve $C$ form a hyper-surface in $M$, consisting of points null separated from $C$.
The assumptions made about the normal bundle of twistor curves in the nonlinear graviton construction ensure that this null separation condition leads to quadratic light cones on~$TM$, and therefore defines a conformal structure. Dropping some of these assumptions leads to {\em causal structures}, where the associated null cones are no-longer quadratic, but instead of higher order \cite{HS1, HS2, krynski, omid}. In this paper I give examples of such
structures which result from path geometries on two-dimensional
surfaces. These geometries, as well as their duals, generalise the Klein duality between ${\RP^2}$ and ${\big(\RP^2\big)}^*$.

\section{Summary of the results}
\subsection{Path geometry}
A path geometry $\u$ on a two-dimensional surface $\U$ is a family of unparametrised smooth curves, from now on called paths, one through each point of $\U$ in each direction.
If the paths of $(\U, \u)$ are unparametrised geodesics of a torsion-free
affine connection
$\nabla$,
then the path geometry is called {\em projective}. The corresponding projective
structure is an equivalence class $[\nabla]$ of torsion-free affine
connections sharing their unparametrised geodesics with $\nabla$. If at least one connection in $[\nabla]$ is
the Levi-Civita connection of some (pseudo) Riemannian metric $g$, then the corresponding path geometry is called {\em metrisable}. The projective path geometries were invariantly characterised more than a century ago
\cite{cartan,liouville}, but the necessary and sufficient conditions for a path geometry to be metrisable were only given relatively recently \cite{BDE}.

 A {\em dual path geometry} $(\Z, \z)$ is a surface $\Z$ whose points correspond to the paths of $\U$. The paths of $\z$ correspond to points of $\U$. This gives rise to a double-fibration picture \cite{bryant, cartan}
\be\label{double_fib}
\U\stackrel{\mu}\longleftarrow
{\mathcal I}\stackrel{\nu}\longrightarrow {\Z},
\ee
where $\I=\{(u\in \U, z\in \Z), u\in z\}\subset
\U\times \Z$ is the three-dimensional incidence space. Both projections in (\ref{double_fib}) are submersions, and for
all incident pairs $(u, z)\in \I$ the curves $\mu^{-1}(u)$ and~$\nu^{-1}(z)$
meet transversally, and their
tangents span a two-plane in $T_{(u,z)}\I$ which gives rise to a~contact structure on $\I$.

\subsection{Dancing pairs}
Let $M=(\U\times \Z)\setminus\I$ be the four-dimensional manifold of non-incident pairs.
A point in $m\in M$ consists of a point $u$ in $\U$ and a path $z$ of $\u$ which does not contain $u$.
\begin{Definition}
\label{defi1}
Two points $(u_1, z_1)$ and $(u_2, z_2)$ in $M$ are dancing if there exists a path in $(\U, {\u})$ which contains
 three points: $u_1$, $u_2$, and the intersection $z_1\cap z_2$.
\end{Definition}
 The terminology is taken from \cite{bor},
where the term {\em dancing} is introduced (and motivated) at an infinitesimal level and directly leads to
the Definition~\ref{definull}. The Definition~\ref{defi1} makes sense both locally, and globally on $M$.
In the latter case the paths $z_1$ and $z_2$ may intersect at more than one point
(e.g., if $\U$ is the round sphere, and the paths are the great circles). In this case we require that there exists a path containing
$u_1$, $u_2$ and at least one point of intersection of~$z_1$,~$z_2$. The dancing condition could of course equivalently have been defined in terms of the dual path geometry $(\Z, \z)$.

The simplest example is provided by the classical projective duality, where
$\U=\RP^2$ and $\Z=\big(\RP^2\big)^*$. Both path geometries
$\u$ and $\z$ are projective, and metrisable (in six different ways) by metrics
of constant curvature.
In this case the dancing condition gives rise
to a~neut\-ral signature conformal structure on $M={\rm SL}(3, \R)/{\rm GL}(2, \R)$ which contains
an anti-self-dual Einstein metric with non-zero Ricci scalar, and the isometry group ${\rm SL}(3, \R)$
\cite{bor,casey, DM}. We~shall review this example in Section~\ref{section2}. In general it is convenient to work
with
the infinitesimal dancing condition \cite{bor}, where
$u_1=u$, $u_2=u+\epsilon \dot{u}+\cdots$ and
$z_1=z$, $z_2=z+\epsilon\dot{z}+\cdots$. The tangent vector $(\dot{u}, \dot{z})\in T_{(u, z)}M$ describes
the direction of a path through~$u$, and the turning point of the
path~$z$ (i.e., the point of intersection of $z$ with a nearby path).
\begin{Definition}
\label{definull}
A tangent vector $(\dot{u}, \dot{z})\in T_{(u, z)}M$
is called
{\em null} if the path through $u$ in the direction of $\dot{u}$ contains the turning point of $z$.
\end{Definition}

In Section~\ref{section3} we shall show how the infinitesimal dancing condition can put in the form
\be
\label{lagrangian}
{\mathcal L}(u, z, \dot{u}, \dot{z})=0,
\ee
where ${\mathcal L}\colon TM\rightarrow \R$. In the projectively flat case where
$\U=\RP^2$ and $M={\rm SL}(3)/{\rm GL}(2)$ described above
the function ${\mathcal L}$ is homogeneous of degree two in $(\dot{u}, \dot{z})$, and gives a null-geodesic
Lagrangian for a conformal structure on $M$. One may wonder whether other examples of path geometries
$(\U, \u)$ give rise to conformal structures. The answer to this is negative,
and is provided by the following local (and therefore also global) rigidity
theorem
\begin{Theorem}
\label{theo_rigid}
If the dancing condition on a path geometry $\U$ defines a conformal structure
on $M=(\U\times \Z)\setminus {\mathcal I}$, then $\U=\RP^2$ with its flat projective
structure, and the conformal structure is represented by the
${\rm SL}(3)$-invariant anti-self-dual Einstein metric on $M={\rm SL}(3)/{\rm GL}(2)$.
\end{Theorem}

We shall prove this theorem in Section~\ref{section3} by twistor methods.
In Section~\ref{section4} we shall consider
a path geometry arising from an ${\rm SL}(2, \R)$ invariant projective structure. We~shall prove
\begin{Theorem}\label{theo_2}
Let $\U=\R^2$ and let $\Z$ be the set of ellipses in $\U$ centred at the origin and of area $\pi$.
The dancing condition on $M=(\U\times \Z)\setminus\I$ defines an ${\rm SL}(2, \R)$-invariant sextic
\eqref{lagrangian} on $TM$ which is quartic in $\dot{u}$ and quadratic in $\dot{z}$.
\end{Theorem}

The path geometry $(\U, \u)$ from Theorem~\ref{theo_2} is projective and metrisable. The dual path geometry is not
projective.

Finally in Section~\ref{section5} we shall move away from path geometries, and instead consider
a seven-dimensional
manifold $N\subset \PP^2\times \PP^5$ consisting of non-incident pairs $(a, A)$, where
$a\in \PP^2$ is a~point, and $A\subset \PP^2$ is an irreducible conic not containing $a$. We~shall say
that two pairs $(a, A)$ and $(b, B)$ of points-conics are dancing if there exists an irreducible conic
containing $(a, b)$ and the four points $A\cap B$ of intersections between $A$ and $B$. We~shall prove
\begin{Theorem}
\label{theo3}
Let $A$, $B$ be $3$ by $3$ matrices representing non-singular conics in $\RP^2$, and let~$a$,~$b$ be vectors in $\R^3\setminus\{0\}$ representing points in $\RP^2$.
Two pairs $(a, A)$ and $(b, B)$ are dancing iff
\[
\big(a^{\rm T} A a\big)\big(b^{\rm T}B b\big)-\big(a^{\rm T}Ba\big)\big(b^{\rm T} A b\big)=0.
\]
The infinitesimal dancing condition gives rise to a
conformal structure on $TN$ of signature $(+ + - - 0 0 0 )$
which is degenerate along the fibres of a projection
$N\rightarrow \big(\PP^2\times \big(\PP^2\big)^*\big)\setminus {\I}$ sending the conic
$A$ to the polar line of $(a, A)$.
\end{Theorem}

\section[Dancing on RP\textasciicircum{}2]{Dancing on $\boldsymbol{\RP^2}$}\label{section2}
Let $\big(\U=\RP^2, \u\big)$ be a path geometry corresponding to a flat projective
structure on the projective plane. The dual path geometry
$\big(\Z=\big(\RP^2\big)^*, \z\big)$ is also flat, and the double fibration~(\ref{double_fib}) is just the classical Klein projective duality between
lines in $\RP^2$ and points in $\big(\RP^2\big)^*$. In this section we shall
represent a point $u\in\U$ by an equivalence class
of vectors $P\in \R^3$, the components of a vector being the homogeneous coordinates. Similarly a point $z\in\Z$ will be represented by a vector $L$ corresponding to a line in~$\RP^2$.

Let $M\subset \PP^2\times {\big(\PP^2\big)}^*$ be set of non-incident pairs
$(P, L)$, where $P\in \PP^2$, and $L\subset \PP^2$ is a line. Two pairs $(P, L)$ and $\big(\tP, \tL\big)$ are null-separated in the sense of the Definition~\ref{defi1}
if there exists
a line which contains the three points $\big(P, \tP, L\cap \tL\big)$. This null condition defines a co-dimension one cone in~$TM$:
generically there is no line through three given points.
The action of ${\rm SL}(3, \R)$ given by $(P, L)\rightarrow \big({\bf A}P, L{\bf A}^{-1}\big)$,
where ${\bf A}\in {\rm SL}(3, \R)$
is transitive on~$M$, and a group stabilising a non-incident pair (point, line)
is ${\rm GL}(2, \R)$ which sits in ${\rm SL}(3, \R)$ as a lower-diagonal block. Therefore~$M$
can be identified with ${\rm SL}(3, \R)/{\rm GL}(2, \R)$.

\begin{figure}[h]\centering
\includegraphics[scale=.27]{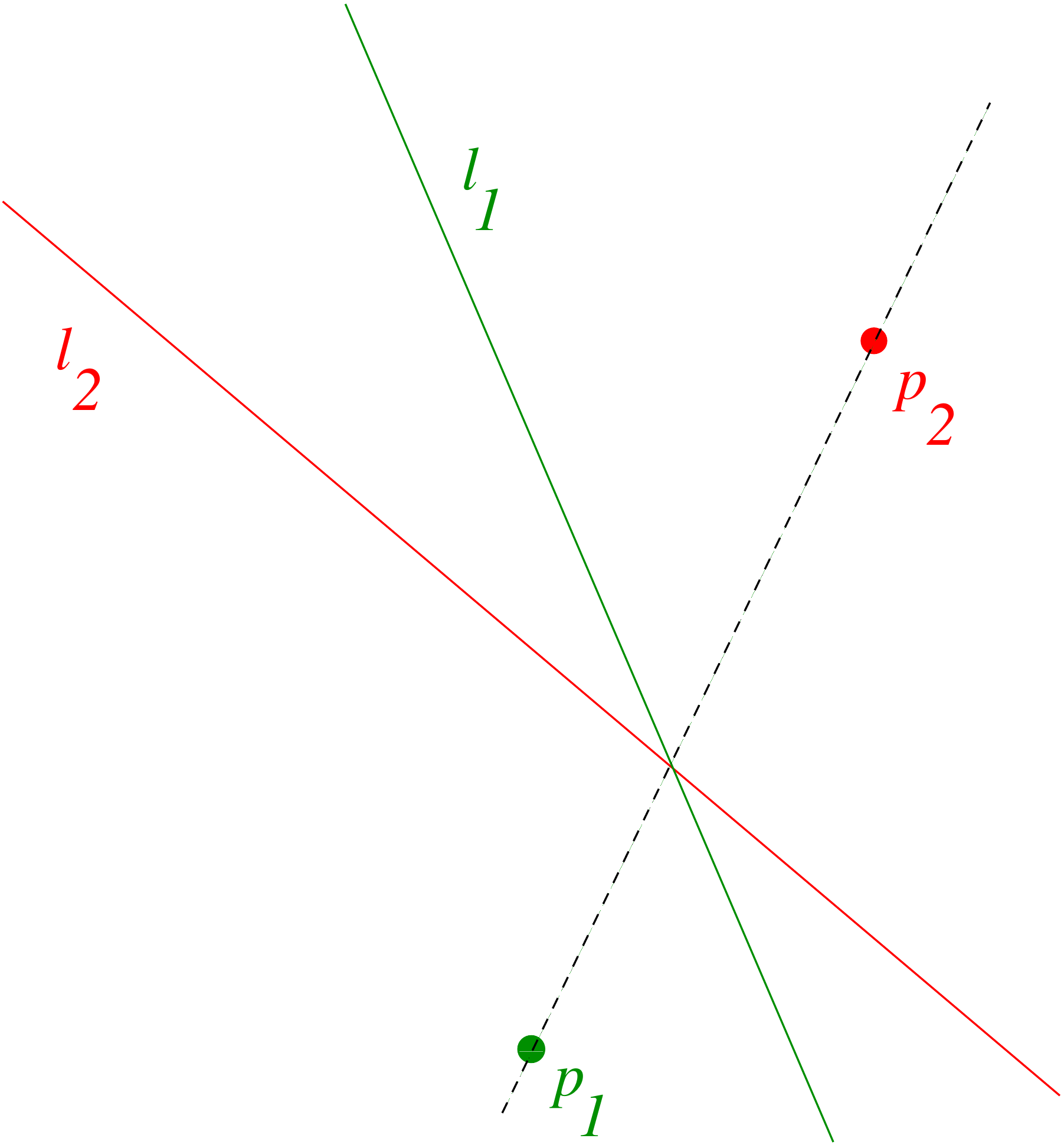}
\caption{Dancing pair.}
\end{figure}

\begin{Proposition}
There exist affine coordinates $\big(x^0, x^1, \zeta_0, \zeta_1\big)$
on $M={\rm SL}(3)/{\rm GL}(2)$ such that two points are dancing if and only if they are null separated
with respect to the anti-self-dual Einsten metric
\be
\label{affine_c}
g={\rm d}\zeta_0{\rm d}x^0+{\rm d}\zeta_1 {\rm d}x^1+(\zeta\cdot {\rm d}x)^2, \qquad \text{where}\quad
\zeta\cdot {\rm d}x\equiv
\zeta_0 {\rm d}x^0+\zeta_1 {\rm d}x^1.
\ee
\end{Proposition}

\begin{proof}
To find an analytic expression for the conformal structure
resulting from the dancing condition
consider two pairs $(P, L)$ and $\big(\tP, \tL\big)$
of non-incident points and lines. Let $L+t\tL$ be a pencil of lines.
There exists $t$ such that
\be\label{dm1}
P\cdot \big(L+t\tL\big)=0, \qquad \tP\cdot \big(L+t\tL\big)=0.
\ee
Eliminating $t$ from (\ref{dm1}) gives
\[
(P\cdot L)\big(\tP\cdot \tL\big)-\big(\tP\cdot L\big)\big(P\cdot \tL\big)=0.
\]
Setting $\tP=P+{\rm d}P$, $\tL=L+{\rm d}L$ yields a metric
$g$ representing the conformal structure
\[
g=\frac{{\rm d}P\cdot {\rm d}L}{P\cdot L}-\frac{1}{(P\cdot L)^2}(L\cdot {\rm d}P)(P\cdot {\rm d}L).
\]
We can use the normalisation $P\cdot L=1$, so that $P\cdot {\rm d}L=-L\cdot {\rm d}P$,
and
\[
g={{\rm d}P\cdot {\rm d}L}+(L\cdot {\rm d}P)^2.
\]
In \cite{DM} the affine coordinates $\big(x^0, x^1, \zeta_0, \zeta_1\big)$ were introduced on $M$ such that
\[
P=\big[x^0, x^1, 1\big],\qquad
L=\big[\zeta_0, \zeta_1, 1-x^0\zeta_0-x^1\zeta_1\big]
\]
with a normalisation $P\cdot L=1$.
The resulting metric is~(\ref{affine_c}).
\end{proof}

In \cite{DM} the metric (\ref{affine_c})
was shown to be Einstein and anti-self-dual. Its isometry group is ${\rm SL}(3, \R)$, and the underlying embedding
of $\mathfrak{sl}(3)$ into $\mathfrak{sl}(4)$ was constructed in \cite{gro}.
The metric was analysed (albeit in different coordinates)
in the dancing context in \cite{bor}, and in connection with path geometries in~\cite{casey}.

Let the flag manifold $\I\equiv F_{12}\big(\C^3\big)\in \PP^2\times {\PP^2}^*$ be the set of incident pairs
$(p, l)$, such that $p\cdot l=0$. This is the twistor space \cite{AHS,Ward_cosmo}
of $(M, g)$.
A $\PP^1$ of $(p, l)\in \I$ corresponding to a point $(P, L)$
consists of all lines $l$ through $P$, and all points
$p=l\cap L$:
\[
P\cdot l=0, \qquad p\cdot L=0, \qquad p\cdot l=0.
\]
Let $(P, L)$ and $\big(\tP, \tL\big)$ be null separated. The corresponding lines in $F_{12}$
intersect at a point $(p, l)$ given by
$p=L\wedge \tL$, $l=P\wedge \tP$,
where $\big[L\wedge\tL\big]^{\gamma}=\epsilon^{\alpha\beta\gamma}L_{\alpha}\tL_{\beta}$ etc.
The incidence condition $p\cdot l=0$ now gives the conformal structure~(\ref{dm1}). This is an illustration of Penrose's nonlinear graviton construction~\cite{Penrose}, adapted to the non-zero cosmological constant~\cite{Ward_cosmo}.

\section{Dancing condition for general path geometries}\label{section3}
The incidence relation encoding the double fibration
(\ref{double_fib}) is given by a map $\psi\colon \U\times \Z\rightarrow \R$, where
$\I$ is the inverse image $\psi^{-1}(0)\subset \U\times \Z$.
Let $(u, z)\in M$, i.e., $\psi(u, z)\neq 0$, and let
$(\dot{u}, \dot{z})\in T_{(u, z)} M$. The infinitesimal dancing condition
is equivalent to the existence of a pair $(u^*, z^*)$ such that
\begin{gather*}
\psi(u, z^*)=0, \quad\ \psi(u^*, z)=0, \quad\ \psi (u^*, z^*)=0, \quad\
\dot{u}\frac{\p\psi}{\p u}(u, z^*)=0, \quad\
\dot{z}\frac{\p\psi}{\p z}(u^*, z)=0,
\end{gather*}
where $z^*$ is the path in the direction $\dot{u}$ containing $u$ and
the turning point $u^*$ of the path $z$. Here the notation $\dot{u}\p/\p u$ denotes a differentiation
along a vector field.

We now
eliminate $z^*$ between the 1st and the 4th condition, and eliminate
$u^*$ between the 2nd and the 5th condition. Substituting the resulting
expressions to the 3rd condition gives one relation between
$(u, z, \dot{u}, \dot{z})$ which is of the form (\ref{lagrangian}).

We shall say that a two-dimensional surface $\sigma\subset M$ is totally null
if all tangent vector fields in $T_m\sigma, m\in M$ are null in the sense of
Definition \ref{definull}.
\begin{Lemma}
\label{alpha_s}
There exists a three-parameter family of totally null surfaces
$($call them $\alpha$-surfaces$)$ in $M$ which correspond to points in
${\mathcal I}$. For any $(l, p)\subset {\mathcal I}$ the corresponding
$\alpha$-surface consists of $(L, P)\in M$ such that
\be
\label{alpha_def}
P\subset l, \qquad p\subset L, \qquad P\neq p, \qquad L\neq l.
\ee
\end{Lemma}

\begin{proof} The points in ${\mathcal I}$ clearly
correspond to surfaces in $M$: there is a one-dimensional family of
$P$s satisfying (\ref{alpha_def}), and once $P$ has been fixed, there is a
one dimensional family of $L$s.
To show that these surfaces are null
note that given $(L, P)$ and $\big(\widetilde{L}, \widetilde{P}\big)$ on the $\alpha$-surface $(l, p)$, the points $\big(P, \widetilde{P}, L\cap\widetilde{L}\big)$ all belong to the line $l$ in agreement with the Definition~\ref{defi1}.
\begin{figure}[h]\centering
\includegraphics[scale=.3]{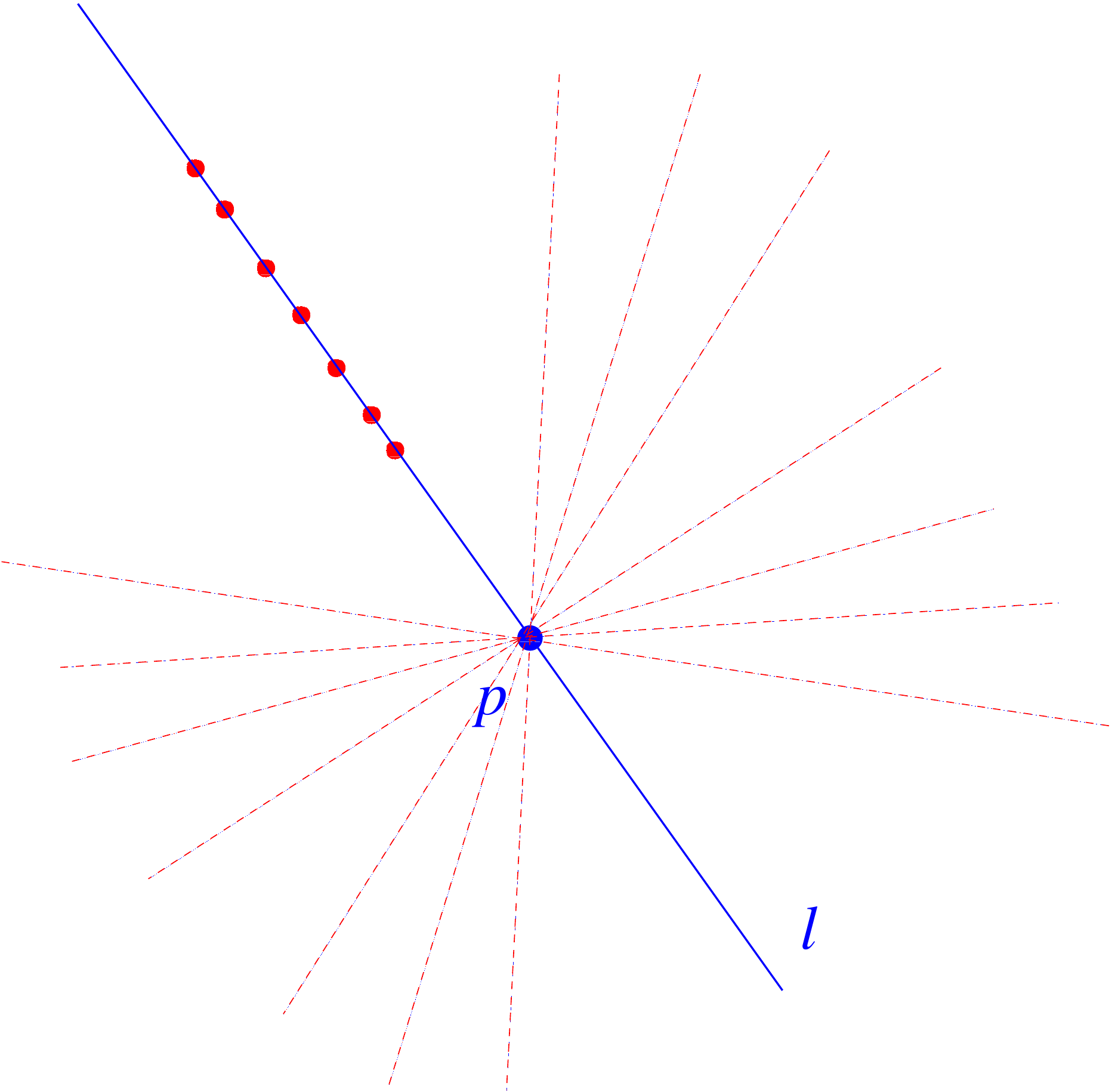}
\caption{$\alpha$-surface.}
\end{figure}
\end{proof}

In addition to the $\alpha$ surfaces $M$ also admits two two-parameter families
of null surfaces (call them $\beta$-surfaces) defined as follows.
The first family $\beta_1$ consist of a fixed point
$P\in \U$ (hence two parameters), and all lines $L$ in $\U$ not containing $P$. The second family $\beta_2$ consist of a fixed line $L\subset \U$ and all points $P$ not on $L$.
If we have instead chosen to model $M$ on the points and paths in $(\Z, \z)$,
then the roles of $\beta_1$ and $\beta_2$ would swap, but the notion
of $\alpha$ surfaces would stay the same. The result of Lemma~\ref{alpha_s} fits into the framework of
{\em half-flat causal structures}~\cite{krynski,omid}, where half-flatness is defined by the existence of a three-parameter family of surfaces whose tangent planes at each point give the ruling for the (not neccesarilty quadratic) null cone at that point. This terminology is motivated by the paradigm example of Penrose~\cite{Penrose}, where the cones are quadratic, and half-flatness condition is equivalent to the anti-self-duality of the Weyl tensor.

We are now ready to prove Theorem \ref{theo_rigid}, and thus establish the uniqueness
of the dancing metric.

\begin{proof}[Proof of Theorem \ref{theo_rigid}]
If there exists a conformal structure on $M$, then, by Lemma \ref{alpha_s}
it is necessarily anti-self-dual (with the right choice of the orientation).
Indeed, there exists a~three-parameter family of null surfaces. All null
surfaces in four dimensions are self-dual (SD) or anti-self-dual (ASD), and we choose an orientation
of $M$ such that the $\alpha$-surfaces are SD. Thus the Weyl
tensor of any metric $g$ in the conformal class $[g]$ is ASD by Penrose's nonlinear graviton theorem
\cite{Penrose} adapted to neutral signature.

Now we shall argue that the conformal structure $[g]$ defined
by the $\alpha$-surfaces of Lemma \ref{alpha_s} must contain an Einstein
metric. To show it, it is enough to demonstrate that the twistor
space~$\I$ of incident pairs ($\alpha$-surfaces) admits a contact structure.
The existence of an Einstein metric will then be a consequence of
the theorem of Ward \cite{Ward_cosmo}. The contact
structure is intrinsically defined by the double-fibration picture
(\ref{double_fib}). The space $\I$ is equipped with two one-dimensional
distributions defining the quotients to $\U$ and $\Z$. The resulting
two-dimensional distribution is non-integrable, and thus it defines
a one-form $\theta$ on $\I$ such that $\theta\wedge {\rm d}\theta\neq 0$.

We have shown (after the proof of Lemma \ref{alpha_s}) that
on top of the SD $\alpha$-surfaces there exist the two, two-parameter
family of $\beta$-surfaces.
Let $(\spp, \epsilon)$, and $(\spp, \epsilon')$ be two symplectic
real vector bundles over $M$ (see \cite{ADM, AHS}). We~shall make use of the two isomorphisms: $TM=\spp\otimes \spp'$, and
${\Lambda^2}_-=\spp\odot \spp$, where~${\Lambda^2}_-$ is the rank-3 vector bundle of
ASD two-forms.
Let the two families of
$\beta$ surfaces correspond to ASD two-forms
$\Sigma_1$ and $\Sigma_2$ in $\Gamma\big({\Lambda^2}_-\big)$, or equivalently to two
sections~$\iota$,~$\tilde{\iota}$ of~$\spp$.

We can rescale $\iota$ so that the corresponding two-form is
closed, and proportional to \mbox{${\rm d}x^1\wedge {\rm d}x^2$} for some functions $\big(x^1, x^2\big)$ on~$M$
which are constant on each $\beta$-surface in the first two parameter family.
The Frobenius integrability conditions now imply the local existence of
two more coordinates such that
$\zeta_0$ and $\zeta_1$ on~$M$ such that $\operatorname{Ker}(\Sigma_1)=
\operatorname{span}\{\partial/\partial \zeta_0, \partial/\partial \zeta_1\}$.
The functions $(\zeta_0, \zeta_1)$ are then the coordinates on the $\beta$-surface.
The corresponding metric takes the form
\be
\label{metric_md}
g={\rm d}\zeta_A\odot {\rm d}x^A+\Theta_{AB}(x, \zeta){\rm d}x^A\odot {\rm d}x^B, \qquad A, B=0, 1
\ee
for some symmetric two-by-two matrix $\Theta$. The anti-self-duality
condition on the Weyl tensor forces the components of $\Theta$ to be at most
cubic in $(\zeta_0, \zeta_1)$, with additional algebraic relations between the components. Imposing the Einstein condition gives
\be
\label{thetaAB}
\Theta_{AB}=\zeta_A\zeta_B+\Rho_{(AB)}-\Gamma_{AB}^C\zeta_C,
\ee
where the coefficients $\Gamma_{AB}^C$ only depend on $\big(x^0, x^1\big)$,
and $\Rho_{AB}$ is the projective Schouten tensor of
affine connection $\nabla$ with connection components~$\Gamma_{AB}^C$.
The $\beta$-distribution defined by~$\Sigma_1$ is parallel, and
\[
\nabla^g \Sigma_1=6{\mathcal A}\otimes \Sigma_1,
\]
where $\nabla^g$ is the Levi-Civita connection of $g$, and the one-form~${\mathcal A}$ is a symplectic connection for the symplectic structure
structure~\cite{DM} given by
\be\label{omega11}
\Omega={\rm d}{\mathcal A}={\rm d}\zeta_A\wedge {\rm d}x^A+\Rho_{AB}{\rm d}x^A\wedge {\rm d}x^B.
\ee
Up to this point our computation has followed \cite{DM}. Now consider the second
family of $\beta$ surfaces given by the ASD two-form $\Sigma_2$, which in our
chosen coordinate system is given by
\[
\Sigma_2= \epsilon^{AB}\big({\rm d}\zeta_A+\Theta_{AC}{\rm d}x^C\big)\wedge \big({\rm d}\zeta_B+\Theta_{BD}{\rm d}x^D\big).
\]
The normalisation condition $\epsilon(\iota, \tilde{\iota})=1$, and
the condition that the distribution defined by $\Sigma_2$ is parallel
together imply that this two-form satisfies
\[
\nabla^g\Sigma_2=-6{\mathcal A}\otimes \Sigma_2,
\]
where ${\mathcal A}$ is given by (\ref{omega11}). We~will need only the skew-symmetric part of
this equation, which gives
\begin{gather*}
0={\rm d}\Sigma_2+6{\mathcal A}\wedge\Sigma_2 =\epsilon^{CD}\nabla_{[A}\Rho_{B]C} {\rm d}x^A\wedge {\rm d}x^B\wedge {\rm d}\zeta_D,
\end{gather*}
where the second equality follows from a rather complicated computation which
we have performed on MAPLE. Therefore this second family of $\beta$-surfaces
exists iff $\nabla_{[A}P_{B]C}=0$, which (for two-dimensional projective structure) is precisely the condition (see, e.g., \cite{BDE}) that the
connection~$\nabla$ is projectively
equivalent to a flat connection, where we can chose
 $\Gamma_{AB}^C=0$ and $\Rho_{AB}=0$. It was shown in~\cite{DM} that
the metric~(\ref{metric_md}) with $\Theta_{AB}$ given by~(\ref{thetaAB})
is invariant under projective changes of connection complemented by an affine
translation of $\xi_A$. Therefore the resulting metric is of the form
(\ref{affine_c}), which completes the proof.
\end{proof}

\section{Dancing ellipses}\label{section4}

\subsection[A dual pair of SL(2)-invariant path geometries]{A dual pair of $\boldsymbol{{\rm SL}(2)}$-invariant path geometries}

Let $\Z$ be the set of ellipses in $\R^2$, centered at the origin and of area $\pi$. Each such ellipse is given uniquely by an equation of the form
\be\label{I}
Ex^2+2Fxy+Gy^2=1,
\ee
where $(E,F,G)\in \R^3$ satisfy
\be\label{Z}
EG-F^2=1, \qquad E>0.
\ee
Thus $\Z$ can be identified with a two-dimensional surface, one sheet of the two-sheeted hyperboloid in $\R^3$ given by equation~\eqref{Z}.

For each $z=(E,F,G)\in \Z$, denote by $\zh\subset \U:=\R^2\setminus \{(0,0)\}$ the corresponding ellipse (the solution space of~\eqref{I} with the given~$E$, $F$, $G$). For example, for $z=(1,0,1)$, $\zh$ is the unit circle centered at the origin.

For each $u=(x,y)\in\U$ let $\uh\subset \Z$ be the set of ellipses that pass through $u$. Using our hyperboloid model for $\Z$, we can see from equation~\eqref{I} that $\uh$ is the intersection of some affine plane in $\R^3$ with $\Z$. For example, for $u=(1,0)$,
$\uh$ is the parabola $G=F^2+1$ in the plane $E=1$.
In fact, the space $\Z$ is the hyperboloid model of the hyperbolic plane and the curves $\uh\subset Z$, $u\in\U$, are the {\em horocycles} of $\Z$ (circles tangent to the real axis, in the Poincar\'e upper half plane model).
Note that the map $u\mapsto\uh$ is $2:1$. Namely, $u$ and $-u$ define the same curve in $\Z$.

In this section we shall find it convenient to use the following form
of the infinitesimal dancing condition (compare Definition \ref{definull})
\begin{Definition}Let $M\subset \U\times\Z$ be the subset of {\em non-incident} pairs $(u,z)$. A tangent vector $(\dot u, \dot z)\in T_{(u,z)}M$ satisfies the {\em dancing condition} if there is an {\em incident} pair $(u^*, z^*)\in\U\times\Z$ such that
\begin{enumerate}\itemsep=0pt
\item[$(1)$] $\dot u$ is tangent to $\zh^*$ at $u$,
\item[$(2)$] $\dot z$ is tangent to $\uh^*$ at $z$.
\end{enumerate}
\end{Definition}

Therefore the vector is null in the sense of Definition~\ref{definull} iff it satisfies the dancing condition.

The manifold $M$ is the disjoint union of two connected components $M=M_{\rm in}\cup M_{\rm out}$, where~$M_{\rm in}$ is the set of pairs $(u,z)$ such that $u$ lies inside of the ellipse $\zh$ (or $z$ lies inside the horocycle~$\uh$) and $M_{\rm out}$ is the set of pairs $(u,z)$ such that $u$ lies outside $\zh$ (or~$z$ lies outside~$\uh$). The group ${\rm SL}(2, \R)$ acts (diagonally) on $M$, preserving each of the two components and the dancing condition on it.

\subsection{Dancing conditons and sextics}
We are now ready to establish Theorem~\ref{theo_2}. We~shall explicitly construct a sextic form
\[
S\in \mbox{Sym}^4(T\U)\otimes \mbox{Sym}^2(T\Z)\subset \mbox{Sym}^6(TM)
\]
such that a vector field $V\in \Gamma(TM)$ is null iff $S(V, V, V, V, V, V)=0$.

\begin{proof}[Proof of Theorem \ref{theo_2}]
We parametrize $\Z$ by the upper half
plane
$H=\big\{(a,b)\in\R^2\st b>0\big\}$. Define $f\colon H\to\Z$ by
\[
f(a,b)=\frac{1}{b}\big(a^2+b^2, -a, 1\big).
\]
Then the incidence condition equation~\eqref{I} can be written as
\be\label{U}
\Phi(x,y,a,b):=\big(a^2+b^2\big)x^2 -2 a x y+y^2-b=0.
\ee
The dancing condition on $(\dot u, \dot z)\in T_{(u,z)}M$
amounts to the existence of a pair $(u^*, z^*)\in\U\times\Z$ satisfying
\begin{align}\label{cC}
&\Phi(u, z^* ) = \Phi(u^*,z^*)=\Phi(u^*,z )= \frac{\partial \Phi}{\partial u}(u,z^*)\dot u=
 \frac{\partial \Phi}{\partial z}(u^*,z)\dot z=0.
\end{align}
Using the coordinates
\[
u=(x,y),\quad\ z=(a,b),\quad\ u^*=(X,Y),\quad\ z^*=(A,B),\quad\ \dot u=(\dot x, \dot y),\quad\ \dot z=(\dot a, \dot b),
\]
these equations are
\begin{gather}
\big(A^2+B^2\big) x^2-2 A xy+y^2-B=0, \nonumber
\\
\big(A^2+B^2\big) X^2-2 A XY+Y^2-B=0, \nonumber
\\
\big(a^2+b^2\big) X^2-2 a XY +Y^2-b=0, \nonumber
\\
\big[\big(A^2+B^2\big) x- A y\big]\dot x+ ( y- A x)\dot y=0, \nonumber
\\
2X( a X- Y)\dot a+\big(2 b X^2-1\big)\dot b =0.\label{cCC}
 \end{gather}

The dancing condition on $TM$ is ${\rm SL}(2, \R)$-invariant, so it is enough to study it along a {\em section} of the
${\rm SL}(2, \R)$-action on $M$, i.e., a subset $\Sigma\subset M$ intersecting every
${\rm SL}(2, \R)$-orbit. We~take $\Sigma=\{x=1, y=0, a=0\}\subset M$, with $b$ a coordinate along $\Sigma$ ($0<b<1$ for $M_{\rm in}$ and $b>1$ for~$M_{\rm out}$). Along $\Sigma$, equations \eqref{cCC} reduce to
\begin{gather*}
A^2+B^2=B, \\
B \big(X^2-1\big)-2 A XY+Y^2=0 \\
b^2 \big(X^2-1\big) +Y^2=0 \\
 B\dot x- A \dot y=0 \\
2X Y\dot a+\big(1-2 b X^2\big)\dot b =0.
 \end{gather*}
We use the 1st and 4th equation to solve for $A$, $B$ and substitute in the other~3 equations, obtaining
\begin{gather}
Y^2 \dot x^2- 2 X Y{\dot x} {\dot y} + \big(X^2+Y^2-1\big){\dot y}^2= 0, \nonumber
\\
b^2 X^2+Y^2-b =0, \nonumber
\\
2 X Y {\dot a} + \big(1-2 b X^2\big){\dot b} = 0.
\label{eq:nullred}
\end{gather}
Now we substitute $Y=PX$ in equations~\eqref{eq:nullred}, use the 2nd obtained equation to solve for $X^2$ and substitute in the other two, obtaining
\begin{gather}
\big[b{\dot x}^2+(b-1) {\dot y}^2\big] P^2-2 b {\dot x} {\dot y}P
 -b(b-1) {\dot y}^2=0,\nonumber
 \\
 \dot b P^2+ 2 b{\dot a}P -b^2{\dot b} =0.\label{eq:red}
\end{gather}
To eliminate $P$ we take the resultant\footnote{
 The resultant of two polynomials of one variable is a polynomial in their coefficients, that vanishes iff they have a common root.
The resultant of $a_0+a_1P+a_2P^2$, $b_0+b_1P+b_2P^2$ is $(a_0b_2-a_2b_0)^2-(a_0b_1-a_1b_0)(a_1b_2-a_2b_1).$
}
of the two quadratic polynomials in $P$ in equations~\eqref{eq:red}, obtaining (after dividing by $b^2$)
\begin{gather}
 b^4 {\dot b}^2 {\dot x}^4
 -4 b^3 {\dot a} {\dot b} {\dot x}^3 {\dot y}
 +2 b^2 \big[((b-2) b-1) {\dot b}^2-2 (b-1) {\dot a}^2\big]{\dot x}^2 {\dot y}^2
 +4 b \big(1-b^2\big) {\dot a} {\dot b} {\dot x}{\dot y}^3\label{sex_gil}
 \\ \qquad
{} +(b-1)^2 \big[(b-1)^2 {\dot b}^2-4b {\dot a}^2\big]{\dot y}^4 =0.
\tag*{\qed}
\end{gather}
\renewcommand{\qed}{}
\end{proof}

The bi-degree of the sextic
(\ref{sex_gil}) has a geometric explanation. For fixed $(u,z)$ and $\dot u$, one draws in $\Z$ the horocycle $\uh$, the point $z$ and the turning point $z^*\in\uh$ corresponding to $\dot u$. There are then exactly two horocycles $\uh_1^*$, $\uh_2^*$ passing through $z$, $z^*$, giving two
tangent directions~$\dot z$ at~$z$.

Similarly, for fixed $(u,z)$ and $\dot z$, there are four turning points $\pm u_1^*, \pm u_2^*\in \zh,$ corresponding to~$\dot z$. For each antipodal pair $\pm u_i^*$ there are two ellipses $\zh^*$ passing through $u$, $\pm u_i^*$ (they correspond to the two intersection points of the horocircles $\uh$, $\uh^*_i$), giving altogether four tangent direction~$\dot u$ at~$u$.

 \subsubsection{2nd order ODEs}
Let the three-dimensional
$\cI\subset \U\times\Z$ be the solution set of equations \eqref{I}. The natural projections $\cI\to\U$, $\cI\to\Z$ define a double fibration (\ref{double_fib})
so that the fibers of one fibration project to paths in the base of the other fibration,
defining a pair of path geometries. This is a curved version of the more familiar
classical point-line projective duality.
A major difference between this curved case and the flat case is
that the symmetry group drops from ${\rm PSL}(3, \R)$ to~${\rm PSL}(2, \R)$.
A path geometry in the $xy$ plane, such as the one given by equations~\eqref{I} and~\eqref{Z} (or~equivalently, equations \eqref{U}),
is determined uniquely by a 2nd order ODE
\[
y''=F(x,y,y').
\]
To find it, one takes the defining equation, say \eqref{U},
and its two derivatives wrt $x$, assuming $y=y(x)$.
Eliminating $a, b$ and solving for $y''$, we obtain
\be\label{ode1}
y''=(xy'-y)^3.
\ee
The dual path geometry, in the $ab$ plane, is obtained by a similar process, for a function $b=b(a)$.
Eliminating $x$, $y$ and solving for $b''$, we obtain
\[
b''=\epsilon\frac{\big(\sqrt{1+b'^2}-\epsilon\big)\big(1+b'^2\big)}{b}, \qquad \epsilon=\pm 1
\]
(so there are in fact two equations).
These are two examples of (dual) 2D path geometries with a three-dimensional
(local) group of symmetries, classified in 1896 by Tresse \cite[p.~76]{Tr}.

The equation \eqref{ode1} is cubic in $y'$, as so it defines a projective structure. This projective structure is metrisable (see \cite{BDE}) as the integral
curves of \eqref{ode1} are unparametrised geodesics
of a~(pseudo) Riemannian metric. There is in fact a three parameter family projectively equivalent metrics with these paths as unparametrized geodesics. We~can arrange for one of these metrics to be invariant under rotations about the origin in the $xy$ plane. In polar coordinates it is given~by
\[
{\rm d}s^2=\frac{{\rm d}r^2}{(1+r^4)^2}+\frac{r^2{\rm d}\theta^2}{1+r^4}, \qquad
\text{where}\quad x+{\rm i}y=r{\rm e}^{{\rm i}\theta}.
\]

\section{Dancing conics}
\label{section5}
Let $N\subset\PP^2\times\PP^5$ be a set of non-incident pairs $(a, A)$, where
$a\in \PP^2$ and $A$ is an irreducible conic. We~say that two pairs
$(a, A)$ and $(b, B)$ are {\em dancing} if there exists a conic $C$ which passes
through six points $(a, b, A\cap B)$ in $\PP^2$, where $A\cap B$ are the four
intersections of the conics $(A, B)$.
Note that
the dancing condition on $N$ is ``similar'' to that given in Definition~\ref{defi1} for path geometries, in that both conditions
define co-dimension one cones in $TN$ and $TM$ respectively: generically there is no line
through a three given points, and there is no conic through a six given points.
\begin{figure}[h]\centering
\includegraphics[scale=.3]{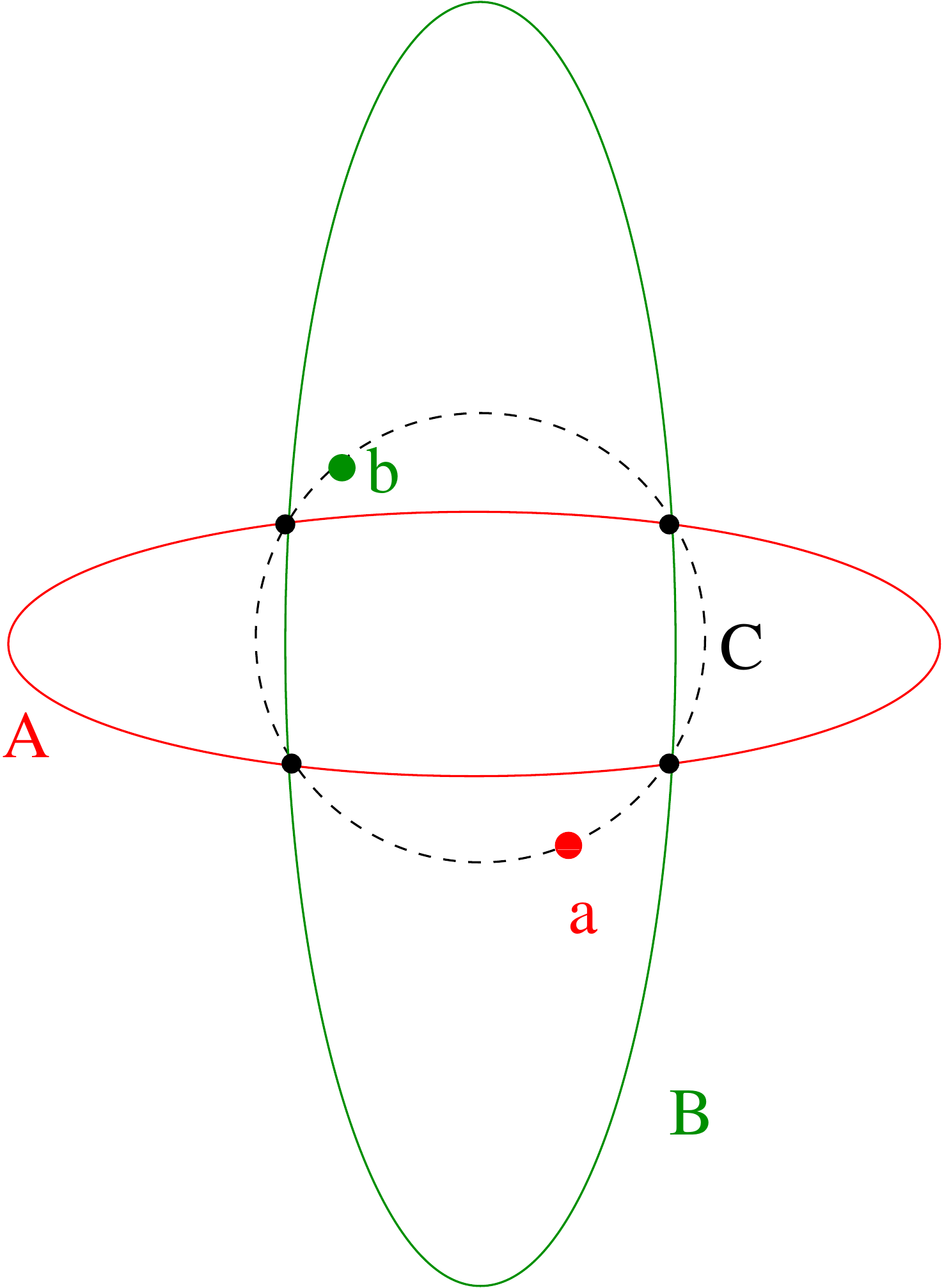}
\caption{Dancing conics.}
\end{figure}

\begin{proof}[Proof of Theorem \ref{theo3}]
The dancing condition
is
\be\label{dancing_f}
\big(a^{\rm T} A a\big)\big(b^{\rm T}B b\big)-\big(a^{\rm T}Ba\big)\big(b^{\rm T} A b\big)=0,
\ee
where (by a slight abuse of notation) $A$ above is a symmetric matrix
representing a conic, and $a^{\rm T}=\big[a^1, a^2, a^3\big]$ is a vector representing the homogeneous coordinates of $a\in\PP^2$. To prove~(\ref{dancing_f}) note that the
conic $C$ must belong to the pencil $(A+tB)$, and contain the points $(a, b)$ so that
\[
a^{\rm T}(A+tB)a=0, \qquad b^{\rm T}(A+tB)b=0.
\]
Elliminating $t$ between these two equations yields (\ref{dancing_f}).

Let $\pi\colon N\rightarrow M$ be a projection given by $\pi(a, A)=(a, \alpha)$,
where $\alpha$ is the polar line of $A$ with respect to $a$.
The fibres of $\pi$ are three-dimensional. For example
if $a$ is a point in an affine space, and $\alpha$ is a line at infinity, then the fiber
of $(a, \alpha)$ consists of all ellipses centred at $a$.

Consider the dancing condition (\ref{dancing_f}) on $N$, and
set $B=A+\epsilon\dot{A}$, and $b=a+\epsilon\dot{a}$, where $\epsilon$ is
small. Substituting this in (\ref{dancing_f}), and keeping the
the second order in $\epsilon$ yields a quadratic condition
\be
\label{ini_gil}
\big(a^{\rm T}\cdot A\cdot a\big)\big(a^{\rm T}\cdot\dot{A}\cdot\dot{a}\big)-\big(a^{\rm T}\cdot\dot{A}\cdot a\big)\big(a^{\rm T}\cdot A\cdot \dot{a}\big)
\ee
(where $\cdot$ denotes the matrix multiplication). The corresponding (degenerate) conformal
structure on $M$ is
\[
G=\big(a^{\rm T}\cdot A\cdot a\big)\big(a^{\rm T}\cdot d{A}\odot d{a}\big)-\big(a^{\rm T}\cdot d{A}\cdot a\big)\odot \big(a^{\rm T}\cdot A\cdot d{a}\big).
\]
We can use inhomogeneous coordinates on $N$ with the pair
$(a, A)$ represented by $a=(x, y, 1)$ and the symmetric matrix $A$ with
$A_{33}=1$. The ${\rm SL}(3, \R)$ invariance alows us without loss of generality to take
the pair $(a, A)$ with $ a_0=(0,0,1)^{\rm T}$, $A_0=\operatorname{diag}(-1,-1,1)$. Then (\ref{ini_gil}) reduces to the quadric
\[
\dot x {\dot A}_{13}+\dot y {\dot A}_{23},
\]
with signature $++--000$, as stated, and
with kernel given by a tangnt subspace
 $\dot x={\dot A}_{13}=\dot y= {\dot A}_{23}=0$.
Now the tangent to the fiber of $\pi$ at $(a_0, A_0)$
is the kernel of the derivative of $\pi$ at $(a_0,A_0)$.
Since $\pi(a, A)=\big(a,a^{\rm T}A\big)$, the derivative at $(a_0, A_0)$
is \[(\dot a,\dot A)
\mapsto \big(\dot a^{\rm T} A_0+ a_0^{\rm T}\dot A, \dot a\big),\]
hence the kernel is given by $\dot Aa_0=\dot a=0,$ or $\dot x={\dot A}_{13}=\dot y= {\dot A}_{23}=0$ as stated.
\end{proof}

\begin{Remark}
One can check that the dancing condition on $N$ is {\em not} the pull back of the dancing condition on $M$, although both are ${\rm SL}(3, \R)$-invariant and horizontal w.r.t.\ $\pi\colon N \to M$. If it were, then
if $(a, A)$ and $(b, B)$ are dancing in $N$, then $\pi(a, A)$, and
$\pi(b, B)$ are dancing in~$M$.
 We shall now argue that this is not the case.
To see it, note that in homogeneous coordinates the projection
$\pi$ is given by $\big(a^i, A_{jk}\big)\rightarrow \big(a^i, \alpha_j=A_{jk}a^k\big)$.
Let $(a, A)$ and $(b, B)$ belong to~$M$.
The intersection of polar lines
$a^{\rm T} A z=0$, $b^{\rm T} B z=0$
is $c$, where $c^i=\epsilon^{ijk}A_{jp}B_{kq}a^pb^q$. This point is co-linear with
$(a, b)$ if $\epsilon_{ijk}c^ia^jb^k=0$. Using the identity
$\epsilon_{ijk}\epsilon^{imn}={\delta_j}^m{\delta_k}^n-{\delta_j}^n{\delta_k}^m$
yields
\[
\big(a^{\rm T} A a\big)\big(b^{\rm T}B b\big)- \big(a^{\rm T} B b\big)\big(b^{\rm T}Aa\big)=0,
\]
which is different than (\ref{dancing_f}).
\end{Remark}

\subsection*{Acknowledgements} This project originated from discussions
with Gil Bor. I am very grateful to Gil for explaining the work~\cite{bor}
to me, and for sharing his geometric insight on Kepler and Hook ellipses.
I thank the anonymous reviewers for their careful reading of the manuscript and many insightful suggestions. I also thank the Mathematics Research Center (CIMAT) in Guanajuato
for hospitality, when some of this research was done.
My research has been
partially supported by STFC grants ST/P000681/1,
and ST/T000694/1.

\pdfbookmark[1]{References}{ref}
\LastPageEnding

\end{document}